\documentclass[a4paper,11pt]{article}
\usepackage{latexsym}
\usepackage{graphicx}
\usepackage{amsfonts}
\usepackage{amsmath} 
\usepackage{amssymb}
\usepackage{amsthm}

\usepackage[dvips, unicode]{hyperref}   
\hypersetup{pdftitle=Density not realizable as the Jacobian determinant of a bilipschitz 
map}
\hypersetup{pdfauthor=Vojt\v{e}ch Kalu\v{z}a}

\author{Vojt\v{e}ch Kalu\v{z}a\thanks{The author was supported by the grant CE-ITI (P202/12/G061) of the Czech Science Foundation and by the grant SVV-2015-260223.}\\
Department of Applied Mathematics\\
Charles University\\
Malostransk\'e n\'am.~25\\
118~00~~Praha~1, Czech Republic}
\title{Density not realizable as the Jacobian determinant of a bilipschitz 
map\thanks{A preliminary version appeared in Czech as a part of
the author's bachelor thesis~\cite{VK} at the Charles University in Prague in~2012.}}

\newcommand{\norm}[1]{\left\|#1\right\|} 
\newcommand{\lnorm}[2]{{\left\|#2\right\|_#1}}
\newcommand{\abs}[1]{\left| #1 \right|} 
\newcommand{\R}{\mathbb{R}}
\newcommand{\Z}{\mathbb{Z}}
\newcommand{\N}{\mathbb{N}}
\newcommand{\edgeab}{\overline{\mathbf{a}\mathbf{b}}}
\newcommand{\edge}[2]{\overline{\mathbf{#1}\mathbf{#2}}}
\newcommand{\rest}[1]{|_{#1}} 
\newcommand{\mb}[1]{\mathbf{#1}}
\newcommand{\set}[2]{\left\{#1, \ldots, #2\right\}}
\DeclareMathOperator{\jac}{Jac}
\DeclareMathOperator{\len}{length}
\DeclareMathOperator{\vl}{vl}

\newtheoremstyle{defnodot}
	{\topsep}
  {\topsep}
  {\normalfont}
  {}
  {\bfseries}
  {.}
  {5pt plus 1pt minus 1pt}
  {\thmname{#1}\thmnumber{ #2}\thmnote{ (#3)}}

\newtheoremstyle{themnodot}
	{\topsep}
  {\topsep}
  {\itshape}
  {}
  {\bfseries}
  {.}
  {5pt plus 1pt minus 1pt}
  {\thmname{#1}\thmnumber{ #2}\thmnote{ (#3)}}

\theoremstyle{themnodot}
\newtheorem*{definition*}{Definition} 
\newtheorem{definition}{Definition}

\newtheorem*{theorem*}{Theorem}
\newtheorem{theorem}[definition]{Theorem}
\newtheorem*{lemma*}{Lemma}

\newtheorem*{proposition*}{Proposition}

\newtheorem*{obs*}{Observation}
\newtheorem{obs}[definition]{Observation}
\newtheorem*{fact*}{Fact}

\begin{document}
\maketitle

\begin{abstract}
Are every two separated nets in the plane bilipschitz equivalent?
In the late 1990s, Burago and Kleiner and, independently, McMullen 
resolved this beautiful question negatively. Both solutions are
based on a construction of a density function that is not realizable
as the Jacobian determinant of a bilipschitz map. 
McMullen's construction is simpler than the Burago--Kleiner one, and
we provide a full proof of its nonrealizability, which has not been
available in the literature.
\end{abstract}

\section{Introduction}\label{sec:intro}

\paragraph{Non-equivalent separated nets and nonrealizable density.}

We recall that a \emph{separated net} in the plane is a set $P\subset\R^2$
in which every two points have distance bounded below by some $r>0$ and the distance
between any point in $\R^2$ and the set $P$ is bounded above by another constant $R>0$.
A simple example of a $1$-separated $1$-net is the integer lattice~$\Z^2$.

The following fascinating question was first mentioned by Furstenberg
in the 1960s and it appears in Gromov's book \cite{Grom}:
\emph{Are every two separated nets in the plane bilipschitz 
equivalent?}\footnote{Let $A,B\subseteq\R^2$. We recall that
 a map $\varphi\colon A\to B$ is $L$-\emph{Lipschitz}, for
a real number $L>0$, if $\norm{\varphi(\mathbf{a})-\varphi(\mathbf{b})}\leq L \norm{\mathbf{a}-\mathbf{b}}$ for every $\mathbf{a},\mathbf{b}\in A$. 
We say that $\varphi$ is $L$-\emph{bilipschitz} if 
both $\varphi$ and $\varphi^{-1}$ are $L$-Lipschitz, and
we call $\varphi$ \emph{Lipschitz} or \emph{bilipschitz} if 
it is $L$-Lipschitz or $L$-bilipschitz, respectively,
for some $L>0$. Two separated nets $P$ and $Q$ are bilipschitz equivalent
if there is a bilipschitz bijection $\varphi\colon P\to Q$.}

It was resolved negatively in the late 1990s by
Burago and Kleiner~\cite{BK1} and, independently, by
McMullen~\cite{M}. Both of the counterexamples are 
based on constructing a bounded density function
$\rho\colon\R^2\to\R$ with $\inf\rho>0$ that is not realizable
as the Jacobian of a bilipschitz map. That is, there is no
bilipschitz $\varphi\colon\R^2\to\R^2$ such that $\jac(\varphi)=\rho$
holds almost everywhere (a.e.), where $\jac(\varphi)$ is the determinant
of the Jacobian matrix of $\varphi$. (McMullen also showed
that the existence of such a $\rho$ is actually \emph{equivalent}
to the existence of two non-equivalent separated nets.)

According to Burago and Kleiner, the problem of density not
realizable as the Jacobian of a bilipschitz map was first proposed 
by Moser and Reimann in the~60's. Later, Dacorogna and Moser~\cite{DMo} 
showed that for $\alpha\in(0,1)$, every $\alpha$-H\"older function is locally 
the Jacobian determinant of a~$C^{1,\alpha}$ homeomorphism, and they 
posed the question of whether every continuous function is locally
the Jacobian of a~$C^1$ diffeomorphism. Several other authors studied
the problem of prescribed Jacobian in different settings,
for example Ye~\cite{Ye} in Sobolev spaces.

McMullen's construction is simpler and easier to describe than
the Burago--Kleiner one. But while Burago and Kleiner provide
a~complete proof of the nonrealizability of their construction,
McMullen's construction and its proof are only sketched in four short paragraphs,
with a remark that a~detailed proof can be given along the lines of
the Burago--Kleiner proof. 

The author of this note, as a part of his bachelor thesis~\cite{VK},
tried to adapt the Burago--Kleiner
proof to McMullen's construction, but found this less than straightforward,
and ended up modifying the Burago--Kleiner
technique, introducing  additional tricks, and
adjusting numerical parameters of the construction.
Thus, for the sake of future researchers interested
in the details of McMullen's construction, it seems worth publishing
a complete proof.

\begin{figure}[htb]
\centering
\includegraphics[width=120mm]{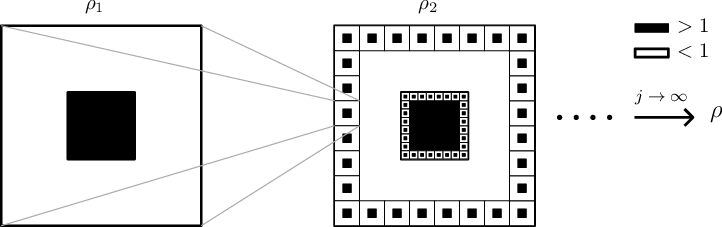}
\caption{\footnotesize{The first two steps of McMullen's construction.}}
\label{fig:1}
\end{figure}

\paragraph{McMullen's construction.}
The nonrealizable function $\rho$ is constructed on the
unit square $S:=[0,1]^2$, as the limit of a sequence 
$\rho_1,\rho_2,\ldots$ of functions, where $\rho_{j}$ is obtained
from $\rho_{j-1}$ by a suitable modification.

To define $\rho_1$, we choose a square $T$ of side $\delta>0$ at the center 
of~$S$ ($\delta$ is one of the parameters of the construction), as in
Figure~\ref{fig:1} left. We define $\rho_1$ as a constant $t_1>1$ on $T$
and as another constant $s_1<1$ on $S\setminus T$. Here $s_1,t_1$ are chosen
so that, first, the average of $\rho_1$ over $S$ is $1$, and second,
the image of~$T$ under any bilipschitz map with Jacobian $\rho_1$ 
has area at least $1-\gamma$, where $\gamma>0$ is another parameter
of the construction. The value of $\gamma$ is chosen small, and thus the image
of $T$ occupies most of the image of~$S$.
McMullen chose $\delta=1/3$ and $\gamma=0.01$ in his sketch, but we will need different values.

To construct $\rho_2$ from $\rho_1$, 
we cover each edge of the squares $S$ and $T$ from inside with much smaller 
squares; those along the edges of $S$ have sidelength $h_2$,
while those along the edges of $T$ have sidelength $\delta h_2$, with $h_2>0$
sufficiently small.
We denote the collection of these new squares by $\mathcal{S}_2$. 

On every square $S'\in \mathcal{S}_2$, we define $\rho_2$ in the same
way as $\rho_1$ was defined on $S$. That is, we consider a smaller
square $T'$ concentric with $S'$ of side $\delta$-times the side of $S'$,
and we set $\rho_2=t_1$ on $T'$ and $\rho_2=s_1$ on $S'\setminus T'$;
see Figure~\ref{fig:1} right. We write $\mathcal{T}_2$ for the collection of the squares $T'$.

On the rest of $S$, similar to $\rho_1$,
the function $\rho_2$ attains a value $t_2$ on the part of $T$ not covered
with $\mathcal{S}_2$, and another value $s_2$ on the part of
$S\setminus T$ not covered with $\mathcal{S}_2$. However, the values of $s_2$ and $t_2$
are slightly different from $s_1$ and $t_1$. Their precise values are determined by
two properties that we want $\rho_2$ to satisfy. Namely, we first choose $t_2$ so that the area of
the image of $T$ under a bilipschitz map with Jacobian $\rho_2$ equals exactly $1-\gamma$, and then
we choose $s_2$ so that the average value of $\rho_2$ on $S$ is exactly $1$.

The construction of $\rho_j$ from $\rho_{j-1}$ follows the same pattern.
We choose $h_j$ sufficiently small. Then we cover the edges of each $S'\in \mathcal{S}_{j-1}$ and of the
corresponding $T'\in\mathcal{T}_{j-1}$ from inside with much smaller squares
forming a collection $\mathcal{S}_{j}$. The sidelengths of these new squares are determined by $h_j$; nevertheless, we cannot say that they are all equal to $h_j$, since the squares being covered have different sidelengths. Instead, we require that the number of squares from $\mathcal{S}_j$ covering an edge of a square in $\mathcal{S}_{j-1}\cup\mathcal{T}_{j-1}$ is $h_{j-1}/h_j$. This implies that $h_j$ is the sidelength of the largest squares in $\mathcal{S}_j$.

We define $\rho_j$ on each square of $\mathcal{S}_{j}$
in the same way as $\rho_1$ was defined on $S$, and we also modify the
values of $\rho_{j-1}$ on each $S'\in \mathcal{S}_{j-1}$
in the same way as was described above for $\rho_2$ on~$S$.
More precisely, for every $S''\in\mathcal{S}_j$ we introduce a smaller
square $T''$ concentric with $S''$ of sidelength $\delta$-times the side of $S''$; these squares $T''$ form a collection $\mathcal{T}_j$. Then we define $\rho_j$ as the constant $t_1$ on $T''$ and as the constant $s_1$ on $S''\setminus T''$. On the rest of each $S'\in\mathcal{S}_{j-1}$, we set $\rho_j$ equal to a constant $t_j$ on the part of $T'$ not covered with $\mathcal{S}_j$ and equal to another constant $s_j$ on the part of $S'\setminus T'$ not covered with $\mathcal{S}_j$. Again, the precise values of $s_j,t_j$ are chosen so that, first, the area of the image of $T'$ under a bilipschitz map with Jacobian $\rho_j$ equals exactly $(1-\gamma)$-times the area of $S'$, and second, the average value of $\rho_j$ on $S'$ is exactly $1$. On the rest of $S$ not contained in $\bigcup\mathcal{S}_{j-1}$, the function $\rho_j$ is equal to $\rho_{j-1}$. That is, during the $j$-th step, we redefine $\rho_{j-1}$ only on the set $\bigcup\mathcal{S}_{j-1}$.

The sequence $h_j$ decreases to $0$ sufficiently fast, namely, so that
$h_{j-1}/h_j\rightarrow\infty$. This ensures that the limit
$\rho=\lim_{j\rightarrow\infty} \rho_j$ is well defined a.e. in $S$, bounded, and
also bounded away from~$0$. This finishes the description
of McMullen's construction.

Now, we are ready to state the theorem which we are going to prove in Section~\ref{sec:proof}:
\begin{theorem}[McMullen]\label{McM}
There exists no bilipschitz map $\varphi\colon S\rightarrow A \subset\R^2$ with $\jac(\varphi) = \rho$ a.e.
\end{theorem}

\paragraph{On differences between the Burago--Kleiner and McMullen's constructions.}
The Burago--Kleiner construction provides a continuous nonrealizable
function, while McMullen's construction sketched above apparently yields only
a~measurable one. 
For explaining the difference, we first describe some of the features of
the Burago--Kleiner construction.

They again work in the unit square $S$.
First, for every $L>1$ and $c>0$, they construct a~measurable function 
$\rho_{L,c}\colon S\rightarrow [1,1+c]$ such that there is no $L$-bilipschitz 
homeomorphism $\varphi\colon S\rightarrow\R^2$ with $\jac(\varphi)=\rho_{L,c}$ a.e. The precise construction of $\rho_{L,c}$, which can be
found in \cite{BK1}, is not important for
us at the moment.

Then they observe that if $\{\rho_{L,c}^k\}_{k=1}^\infty$ is a~sequence of smoothings of~$\rho_{L,c}$ converging to~$\rho_{L,c}$ in~$L^1$, there must be 
some $k_0\in\N$ such that for every $k\geq k_0$, the functions $\rho_{L,c}^k$ are also nonrealizable as Jacobians
of $L$-bilipschitz homeomorphisms, for otherwise,
the Arzel\`a--Ascoli theorem would yield an~$L$-bilipschitz 
homeomorphism $\varphi$ with $\jac(\varphi)=\rho_{L,c}$ a.e. 

Finally, they take a~collection of disjoint squares $S_k\subset S$ converging to a~point $p\in S$, they construct a~new function $\rho\colon S\rightarrow [1,1+c]$ by embedding the function $\rho_{k,\min\{c,\frac{1}{k}\}}$ into $S_k$ for every $k\in\N$, and they define $\rho$ on the rest of~$S$ arbitrarily, while preserving its continuity. Consequently, $\rho$ is continuous and 
it cannot be realized  as the Jacobian of any bilipschitz homeomorphism.

Since we do not know how to prove nonrealizability of McMullen's density
parametrized so that the image of $\rho$ is contained in $(0,1+c]$ with $c>0$
arbitrarily small, we cannot use the method of Burago and Kleiner
outlined above to obtain a continuous version of McMullen's density.
However, it may be possible either to achieve continuity in some other way
without changing the construction too much or devise a better proof.

\section{Preliminaries}

Before we proceed to the proof, we present some definitions and facts.
We denote the $k$-dimensional Lebesgue measure by $\lambda_k$. Since we will deal mainly with the plane, we write just
$\lambda$ instead of $\lambda_2$. We always use the Lebesgue measure unless stated otherwise.

Let $\varphi\colon\R^n\rightarrow\R^n$ be a~map that is (Fr\'echet) differentiable at a point $x\in\R^n$. The matrix consisting of its first partial derivatives at~$x$ is called the \emph{Jacobian matrix} of the map $\varphi$ at the point~$x$.
We denote it by $D\varphi(x)$. The determinant of the Jacobian matrix, $\jac(\varphi)(x):=\det{D\varphi(x)}$, is called the \emph{Jacobian determinant} or simply \emph{Jacobian} of the map $\varphi$ at~$x$. It gives us information about the change of the volume in the neighborhood of~$x$. 

By a~curve we mean an~image of an~interval $I$ under a~continuous map $f$. The length of a~curve $P$ is defined in the usual manner, that is, as $\len(P):=\sup\sum_{k=0}^{n-1}\norm{f(p_{k+1})-f(p_k)}$, where the supremum is taken over all finite partitions of the form $\min I=p_0<p_1<\ldots<p_n=\max I$.

Let us denote the line segment between points $\mathbf{a}$ and $\mathbf{b}$ by $\overline{\mathbf{a}\mathbf{b}}$, the Euclidean distance between $\mb{a}$ and $\mb{b}$ by $\lnorm{2}{\mb{a}-\mb{b}}$, the projection on the first coordinate ($x$-axis) by~$\pi_x$, and the restriction of a~function $f$ to a~set $E$ by~$f\rest{E}$.

We write $\partial E$ for the boundary of a set $E$, that is, the closure of $E$ without the interior of $E$.

\begin{obs}\label{obs:lip}
Let $P$ be the image of an~interval of length~$d$ under an~L-Lipschitz map. Then $\len(P)\leq L d$.
\end{obs}

We leave the easy proof to the reader.

\begin{theorem}\label{jacint}
Let $f\colon A\subseteq\R^k\rightarrow\R^k$ be an injective Lipschitz map. Then for every measurable set $E\subseteq A$, $f(E)$ is also measurable and
\begin{displaymath}
\int_{E}\abs{\jac(f)} d\lambda_k=\lambda_k(f(E)).
\end{displaymath}
\end{theorem}

\begin{proof}
This theorem is a~corollary of the change of variables theorem for Le\-bes\-gue integral and Rademacher's theorem; it can be found in Fremlin's monograph~\cite[Corollary~263F]{Frem}, for example.
\end{proof}

\section{The proof}\label{sec:proof}

In the following, we present the complete proof of Theorem~\ref{McM}, i.e., we prove nonrealizability of the function $\rho$ constructed in Section~\ref{sec:intro}.

Recall that during the $j$-th step of the construction of $\rho$, we introduce a collection of tiny squares $\mathcal{S}_j$ and modify the function $\rho_{j-1}$ on each of the squares from that collection. For each square $S'\in\mathcal{S}_j$ we denote the smaller square placed at the center of~$S'$ by~$T'$. The function $\rho$ has the following properties. For every level $j$ and every $S'\in\mathcal{S}_j$ we have $\int_{S'}{\rho\,d\lambda} = \lambda(S')$ and $\int_{T'}{\rho\,d\lambda} = (1-\gamma)\lambda(S')$.

\begin{proof}[Proof of Theorem~\ref{McM}]
Assume to the contrary that there exists a~bilipschitz homeomorphism $\varphi\colon S\rightarrow\R^2$ with the Jacobian determinant $\jac\varphi=\rho$ a.e.

Let us denote by~$H$ the set consisting of all edges of all covering squares, i.e., the squares in $\bigcup_{j=1}^{\infty}\mathcal{S}_j$, in the construction of the function $\rho$. We set $K:= \sup_{\overline{\mathbf{p}\mathbf{q}}\in H} \frac{\lnorm{2}{\varphi(\mathbf{p})-\varphi(\mathbf{q})}}{\lnorm{2}{\mathbf{p}-\mathbf{q}}}$. Since the map $\varphi$ is bilipschitz, we get $K<\infty$. Now, we choose a~parameter $\alpha:=\alpha(\gamma, \delta, K)>0$ sufficiently small, whose value will be set at the end of the proof. From the definition of the supremum we have that there exists a~covering square $S'$ and one of its edges $\overline{\mathbf{a}\mathbf{b}}$ with $\frac{\lnorm{2}{\varphi(\mathbf{a})-\varphi(\mathbf{b})}}{\lnorm{2}{\mathbf{a}-\mathbf{b}}}>K(1-\alpha)$. We fix this edge $\edgeab$ until the end of the proof.
Without loss of generality, we assume that $\mathbf{a}=(0,0)$, $\mathbf{b}=(b,0)$, $\varphi(\mathbf{a})=(0,0)$, and $\varphi(\mathbf{b})=(b',0)$, with $b>0$ and $b'>K(1-\alpha) b$.

By the construction of~$\rho$, the edge $\overline{\mathbf{a}\mathbf{b}}$ is covered with arbitrarily small squares. In other words, for a~chosen $N_0:=N_0(\alpha, \gamma, \delta, K)$, which will be set at the end of the proof, too, we can find $j\in\N$ and $N\geq N_0$ so that the edge $\overline{\mathbf{a}\mathbf{b}}$ is covered with $N$ squares $S_1, \ldots, S_N\in\mathcal{S}_j$.

These squares form a~tiny long rectangle, which we call~$R$. Let $h$ stand for their sidelength; this means that $h=b/N$.
Inside of every $S_i, i=1,\ldots,N$, we also have the square $T_i\in\mathcal{T}_j$ with $\delta$-times smaller sidelength. For clarity, we add Figure~\ref{fig:2}.
\begin{figure}[htb]
\centering
\includegraphics[width=120mm]{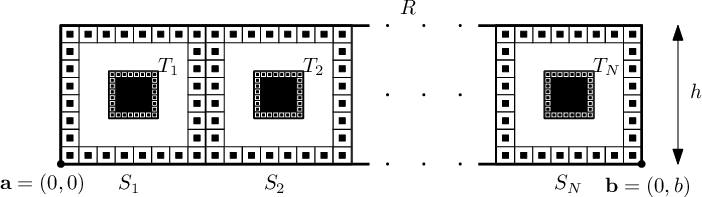}
\caption{\footnotesize{The rectangle~$R$ covering the edge $\overline{\mathbf{a}\mathbf{b}}$.}}
\label{fig:2}
\end{figure} 

The main idea of the proof is that by choosing $\alpha$ to be very small we force $\varphi$ to map the long edges of the rectangle $R$ to almost straight lines stretched by a~factor almost~$K$. The function $\rho$ has been constructed so that most of the mass within $S_i$ is concentrated on~$T_i$. This implies that the majority of each $\varphi(S_i)$ has to be filled up with $\varphi(T_i)$.

On the other hand, the sidelength of each $T_i$ is $\delta$-times smaller than that of~$S_i$. Each side of~$T_i$ is also covered with smaller squares, and thus it cannot be stretched more than by the factor~$K$.
The only way in which all these conditions can be fulfilled is that the images of the long edges of the rectangle $R$ somewhat ripple up above the images of the respective rectangles $T_i$, while between them they have to ripple down. But this forces the images of the long edges of~$R$ to become very long, eventually longer than the constant $K$ allows, which leads to a~contradiction. We illustrate the outlined idea in Figure~\ref{fig:3}.
\begin{figure}[htb]
\centering
\includegraphics[width=120mm]{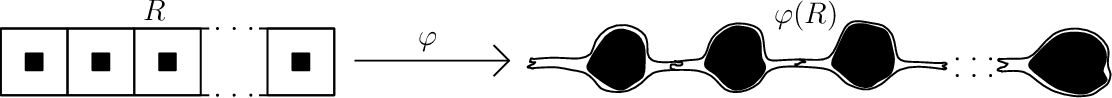}
\caption{\footnotesize{Deformation of the edges of~$R$ under the map $\varphi$.}}
\label{fig:3}
\end{figure}

In order to make this idea rigorous, we look at the change of the length of vertical cuts through the set $\varphi(R)$. The most complicated part of the proof is an~estimation of the length of the boundary of~$\varphi(R)$. Because we have almost no control of~$\varphi$ locally, the shapes of different $\varphi(S_i)$ can be various. We manage the described difficulty by examining the squares $S_i$ in seven-tuples. For this purpose we make an additional technical assumption that the number of squares covering every edge during the construction of~$\rho$ is divisible by seven. Let $R_i$ stand for the rectangle formed by the seven consecutive squares $S_{7i-6},\ldots, S_{7i}$. Without loss of generality, we assume that the edges of the squares $S_i$ covering the edge $\edgeab$ are their bottom edges. Let us write $\mathbf{a}_i$ and $\mathbf{b}_i$ for the left and right vertices at the bottom of the rectangle $R_i$, respectively, and $\mathbf{c}_i$ and $\mathbf{d}_i$ for the left and right vertices on the top of~$R_i$, respectively.

\begin{definition}\label{nicedef}
We call the rectangle $R_i$ with its bottom vertices $\mathbf{a}_i$ and $\mathbf{b}_i$ \textbf{nice} if $\abs{\pi_x(\varphi(\mathbf{a}_i))-\pi_x(\varphi(\mathbf{b}_i))}>K(1-2\alpha)7h$.
\end{definition}

The factor $1-2\alpha$ in the preceding definition is chosen to have a constant fraction of the rectangles $R_i$ \emph{nice}, more precisely, to get the following observation:

\begin{obs}\label{nice}
There are at least $N/14$ \emph{nice} rectangles $R_i$ in $R$.
\end{obs}

\begin{proof}
Let $r$ stand for the number of \emph{nice} rectangles $R_i$. This implies that for $N/7-r$ rectangles $R_i$ we have $\abs{\pi_x(\varphi(\mathbf{a}_i))-\pi_x(\varphi(\mathbf{b}_i))}\leq K(1-2\alpha)7h$. On the other hand, the edges $\overline{\mathbf{a}_i\mathbf{b}_i}$ connect the vertices $\mathbf{a}$ and $\mathbf{b}$. The horizontal distance of the points $\varphi(\mathbf{a})$ and $\varphi(\mathbf{b})$ has been assumed to be greater than $K(1-\alpha)Nh$. Since no edge $\overline{\mathbf{a}_i\mathbf{b}_i}$ can be stretched more than by the factor~$K$, we can calculate a~lower bound on~$r$:
\begin{displaymath}
\left(\frac{N}{7}-r\right)(K(1-2\alpha)7h)+Kr\cdot7h>K(1-\alpha)Nh
\end{displaymath}
Simple calculation yields $r>N/14$.
\end{proof}

Now, we define the set $D_i:=\pi_x(\varphi(\overline{\mathbf{a}_i\mathbf{b}_i}))\cap\pi_x(\varphi(\overline{\mathbf{c}_i\mathbf{d}_i}))$ for every rectangle $R_i$. Next, we define a~function $f_i\colon D_i\rightarrow [0,\infty)$ measuring the length of vertical cuts through the set $\varphi(R_i)$ at a~point $x\in D_i$, i.e., $f_i(x)=\lambda_1(\{y\in\R|\,(x,y)\in\varphi(R_i)\})$.
It follows from the continuity of the map $\varphi^{-1}$ that the function $f_i$ is Lebesgue integrable for every $R_i$.

We would like to stress that the function $f_i$ is measuring only the length of cuts through the set $R_i$, not the entire $R$. Vertical cuts through $R$ can intersect many $R_i$'s.

The following observation will later help us treat each \emph{nice} rectangle separately. It basically says, that the image of every such an~$R_i$ is drawn almost horizontally and its boundary is not ``too wavy''.

\begin{obs}\label{midRi}
We have $\pi_x(\varphi(S_{7i-3}))\cap\pi_x(\varphi(\overline{\mathbf{a}_i\mathbf{c}_i}))=\emptyset$, and symmetrically, $\pi_x(\varphi(S_{7i-3}))\cap\pi_x(\varphi(\overline{\mathbf{b}_i\mathbf{d}_i}))=\emptyset$ for every \emph{nice} rectangle $R_i$.
\end{obs}

\begin{figure}[htb]
\centering
\includegraphics[width=120mm]{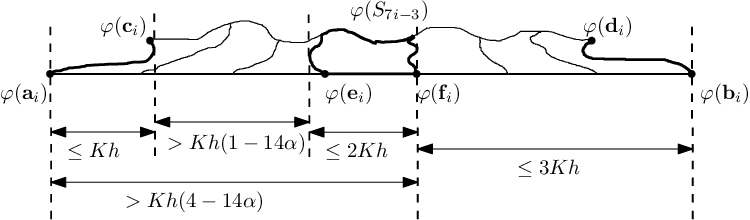}
\caption{\footnotesize{The image of \emph{nice} $R_i$ under $\varphi$.}}
\label{fig:4}
\end{figure}

\begin{proof}
Let $\mathbf{e}_i$ stand for the bottom left vertex of~$S_{7i-3}$ and $\mathbf{f}_i$ for its bottom right vertex. This implies $\lnorm{2}{\mathbf{a}_i-\mathbf{f}_i}=4h$ and $\lnorm{2}{\mathbf{b}_i-\mathbf{f}_i}=3h$. Since the rectangle $R_i$ is \emph{nice}, we also have $\abs{\pi_x(\varphi(\mathbf{a}_i))-\pi_x(\varphi(\mathbf{f}_i))}>K(1-2\alpha)7h-3Kh=Kh(4-14\alpha)$.

The set $\varphi(S_{7i-3})$ has to lie within the circle with radius $2Kh$ centered at~$\varphi(\mathbf{f}_i)$, while the set $\varphi(\overline{\mathbf{a}_i\mathbf{c}_i})$ lies within the circle of radius $Kh$ centered at~$\varphi(\mathbf{a}_i)$. We conclude that the distance between the sets $\pi_x(\varphi(S_{7i-3}))$ and $\pi_x(\varphi(\overline{\mathbf{a}_i\mathbf{c}_i}))$ is at least $Kh(4-14\alpha)-2Kh-Kh=Kh(1-14\alpha)$, which is positive for $\alpha\in(0,1/14)$. This proof is outlined in Figure~\ref{fig:4}.

The second part of the observation is obtained symmetrically.
\end{proof}

The rectangle $R_i$ is formed by the squares $S_{7i-6},\ldots, S_{7i}$. We define the sets $V_i:=\pi_x(\varphi(T_{7i-3}))$ and $C_i:=\pi_x(\varphi(S_{7i-3}))\setminus\pi_x(\bigcup_{j=7i-6}^{7i}\varphi(T_j))$; so we have $V_i\cap C_i=\emptyset$. By Observation~\ref{midRi} it is clear that $\pi_x(\varphi(S_{7i-3}))\subseteq D_i$, and thus $C_i, V_i\subseteq D_i$ for \emph{nice} $R_i$. 

Observation~\ref{midRi} implies that, whenever $R_i$ is \emph{nice}, for every $x\in V_i\cup C_i$, all the points of intersection of the vertical cut at $x$ with $\partial\varphi(R_i)$ also lie on $\partial\varphi(R)$. Indeed, since the map $\varphi$ is a~homeomorphism, the images of the long edges of~$R_i$, which form a~part of $\partial\varphi(R_i)$, are also part of $\partial\varphi(R)$. We will use these points to bound the length of the boundary of $\varphi(R)$.

It is possible that the sets $V_i, V_{i+1}, C_i$, and $C_{i+1}$ are not mutually disjoint. Let $x$ be a common point of $C_i$ and $C_{i+1}$, for example. We already know that, if $R_i$ and $R_{i+1}$ are both \emph{nice}, the points of intersection of $\partial\varphi(R_i)$ and $\partial\varphi(R_{i+1})$ with the vertical cut at $x$ lie on the boundary of $\varphi(R)$. The fact that $\varphi$ is a~homeomorphism implies that all these points are different. This is a~crucial observation in our proof, because it allows us to do the estimates for each of the \emph{nice} rectangles separately and then sum them up.

Let us set $v_i:=\lambda_1(V_i)$ and $c_i:=\lambda_1(C_i)$. Since the length of the edge of~$T_i$ is $\delta h$, we have that the length of the boundary of $\varphi(T_{7i-3})$ is at most $4K\delta h$, which implies $\lambda_1(\pi_x(\varphi(T_{7i-3})))\leq 2K\delta h$, and hence $v_i\leq 2K\delta h$ for every $i \in \set{1}{N/7}$.

For every \emph{nice} $R_i$ we have $\lambda_1(\pi_x(\varphi(S_{7i-3})))>K(1-2\alpha)7h-6Kh=Kh(1-14\alpha)$, and thus $c_i>Kh(1-14\alpha)-7\cdot 2K\delta h=Kh(1-14\alpha-14\delta)$. Because we need $c_i>0$, we have to choose $\delta\in(0, 1/14)$ and $\alpha\in(0,(1-14\delta)/14)$.

By Theorem \ref{jacint} and since we assume $\jac(\varphi)=\rho$ a.e., we have $\lambda(\varphi(T_{7i-3}))=(1-\gamma)h^2$. We define two constants $h_V^i$ and $h_C^i$ denoting the average values of~$f_i$ over the sets $V_i$ and $C_i$, respectively. In other words, the following holds:
\begin{displaymath}
\lambda(\varphi(T_{7i-3}))<\int_{V_i}f_i\,d\lambda_1=:h_V^i\!\cdot\!v_i
\end{displaymath}
\begin{displaymath}
\lambda\left(\bigcup_{j=7i-6}^{7i}\varphi(S_j)\setminus \bigcup_{j=7i-6}^{7i}\varphi(T_j)\right)>\int_{C_i}f_i\,d\lambda_1=:h_C^i\!\cdot\! c_i.
\end{displaymath}
The upper bound $v_i\leq2K\delta h$ yields
\begin{displaymath}
h_V^i>\frac{\lambda(\varphi(T_{7i-3}))}{v_i}\geq\frac{(1-\gamma)h^2}{2K\delta h}=\frac{1}{K}\cdot h\cdot\frac{1-\gamma}{2\delta}.
\end{displaymath}
Using the lower bound $c_i>Kh(1-14\alpha-14\delta)$, we deduce that the following holds for every \emph{nice} $R_i$:
\begin{equation*}
\begin{split}
h_C^i<\frac{\lambda(\bigcup_{j=7i-6}^{7i}{\varphi(S_j)}\setminus\bigcup_{j=7i-6}^{7i}{\varphi(T_j)})}{c_i}&<\frac{7\gamma\cdot h^2}{Kh(1-14\alpha-14\delta)}\\
&=\frac{1}{K}\cdot h\cdot\left(\frac{7\gamma}{1-14\alpha-14\delta}\right).
\end{split}
\end{equation*}

Since $h_V^i$ and $h_C^i$ are the average values of~$f_i$ over $V_i$ and $C_i$, respectively, we get that for every \emph{nice} $R_i$ there must be two points $x_i\in V_i$ and $y_i\in C_i$ such that $f_i(x_i)\geq h_V^i$ and $f_i(y_i)\leq h_C^i$. Thus we have $f_i(x_i)-f_i(y_i)\geq h_V^i-h_C^i$.
Furthermore, we can bound the last term using the bounds on $h_V^i$ and $h_C^i$ derived above. That is, we infer that $h_V^i-h_C^i>\frac{1}{K}\cdot h\cdot\left(\frac{1-\gamma}{2\delta}-\frac{7\gamma}{1-14\alpha-14\delta}\right):=\Delta$. This bound is already independent of $i$. Clearly, $\Delta>0$ if the parameters are chosen appropriately.

Now, we would like to argue that $\Delta$ is the lower bound on the change of height of $\varphi(R)$ over $V_i\cup C_i$ for every \emph{nice} $R_i$. Indeed, it is true that there are two points $\mathbf{u}_i, \mathbf{u}'_i$, the former from the image of the bottom edge $\overline{\mathbf{a}_i\mathbf{b}_i}$ of $R_i$, the latter from the image of the upper edge $\overline{\mathbf{c}_i\mathbf{d}_i}$, such that $\mathbf{u}_i=(x_i,u_i)$, $\mathbf{u}'_i=(x_i,u'_i)$ and $\abs{u'_i-u_i}\geq h_V^i$.
But the same thing about $y_i$ and $h_C^i$ has to be said with a little more care.

The problem is that the vertical cut through $\varphi(R_i)$ at $y_i$ does not have to be connected, i.e., it may consist of several line segments even for \emph{nice} $R_i$. But we know that the length of these line segments is at most $h_C^i$ in total, and hence every line segment of this cut is at most $h_C^i$ long. Consequently, we infer that there are two points $\mathbf{l}_i\in\varphi\left(\overline{\mathbf{a}_i\mathbf{b}_i}\right), \mathbf{l}'_i\in\varphi\left(\overline{\mathbf{c}_i\mathbf{d}_i}\right)$ such that $\mathbf{l}_i=(y_i, l_i)$, $\mathbf{l}'_i=(y_i, l'_i)$ and $\abs{l'_i-l_i}\leq h_C^i$. The situation is depicted in Figure~\ref{fig:6}.

\begin{figure}[htb]
\centering
\includegraphics[width=105mm]{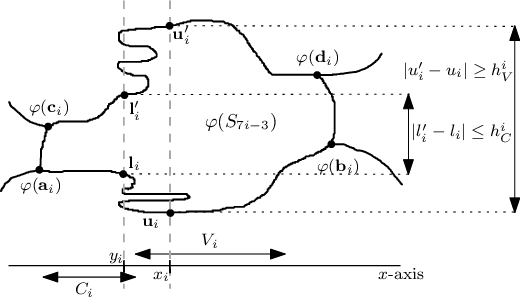}
\caption{\footnotesize{The lower bound on the change of height of $\varphi(R_i)$.}}
\label{fig:6}
\end{figure}

We would like to combine Observation~\ref{nice} and the discussion above to bound below the vertical distance that has to be overcome by $\varphi(R)$. What we mean by this precisely is explained in the following definition:

\begin{definition*}
Let $\Gamma$ be a~Lipschitz map $[0,1]\rightarrow\R^2$ and also a~curve defined by the map. The \textbf{vertical length} of the curve $\Gamma$, denoted by $\vl(\Gamma)$, is defined as $\sup\sum_{k=0}^{n-1}\abs{\Gamma_2(p_{k+1})-\Gamma_2(p_k)}$, where the supremum is taken over all finite increasing sequences $\{p_k\}_{k=0}^n$ of numbers in $[0,1]$ and $\Gamma_2$ denotes the second coordinate function, i.e., the $y$-coordinate of $\Gamma$.
\end{definition*}

In other words, the vertical length of a curve is its length when measuring the distance only in the $y$-coordinate.

Let $J$ be the set of indices $i$ such that $R_i$ is \emph{nice}. By the above discussion, it follows that $\vl(\partial\varphi(R))$ has to be at least $\sum_{i\in J}{\left(\abs{u'_i-l'_i}+\abs{u_i-l_i}\right)}$. Next, we bound it below using the triangle and the reverse triangle inequalities:
\begin{equation*}
\begin{split}
\sum_{i\in J}{\left(\abs{u'_i-l'_i}+\abs{u_i-l_i}\right)} &\geq \sum_{i\in J}\abs{u'_i-l'_i+l_i-u_i}\\
\geq \sum_{i\in J}{\Bigl|\abs{u'_i-u_i}-\abs{l'_i-l_i}\Bigr|} &= \sum_{i\in J}{\left(\lnorm{2}{\mb{u}'_i-\mb{u}_i}-\lnorm{2}{\mb{l}'_i-\mb{l}_i}\right)}.
\end{split}
\end{equation*}

We have chosen the points $\mb{u}'_i, \mb{u}_i, \mb{l}'_i$, and $\mb{l}_i$ so that, for every nice $R_i$, $\lnorm{2}{\mb{u}'_i-\mb{u}_i}-\lnorm{2}{\mb{l}'_i-\mb{l}_i}>\Delta$. By Observation~\ref{nice} we know that there are at least $N/14$ \emph{nice} rectangles $R_i$, and thus $\sum_{i\in J}{\left(\lnorm{2}{\mb{u}'_i-\mb{u}_i}-\lnorm{2}{\mb{l}'_i-\mb{l}_i}\right)}$ is at least $\Delta N/14$.

Let $P$ stand for $\left(\partial\varphi(R)\right)\setminus\varphi(\edge{a}{b})$.
It is easy to see that
\begin{displaymath}
\vl(P)\geq\sum_{i\in J}\abs{u'_i-l'_i},
\end{displaymath}
since the points $\mb{u}'_i, \mb{l}'_i$ are lying on $P$ in order specified by $i$. Now, we aim to obtain a lower bound on the latter quantity.

Before we proceed, we would like to describe the strategy used in the rest of the proof. We know that the curve $\varphi(\edge{a}{b})$ is at most $Kb$ long and connects the points $\varphi(\mb{a})$ and $\varphi(\mb{b})$ at distance almost $Kb$. This means that $\vl(\varphi(\edge{a}{b}))$ has to be very small. On the other hand, we have a lower bound on $\vl\left(\partial\varphi(R)\right)$, which implies a lower bound on $\vl(P)$. Moreover, $P$ also connects the points $\varphi(\mb{a})$ and $\varphi(\mb{b})$; therefore, together with the lower bound on $\vl(P)$, we can calculate a lower bound on $\len(P)$ that becomes large if the parameters are chosen appropriately, eventually larger than the constant $K$ allows. This will be the desired contradiction.

As we already know, the vertical length of the whole boundary of $\varphi(R)$ is at least $\sum_{i\in J}{\left(\abs{u'_i-l'_i}+\abs{u_i-l_i}\right)}$, which in turn is at least $\Delta N/14$. Subtracting the second terms of the sum we get the following inequality:
\begin{equation}
\label{l_bound}
\begin{split}
\frac{\Delta N}{14}-\sum_{i\in J}\abs{u_i-l_i}\leq\sum_{i\in J}\abs{u'_i-l'_i}\leq\vl(P).
\end{split}
\end{equation}

Consequently, in order to get the desired lower bound on $\vl(P)$, it suffices to derive an upper bound on the quantity $\sum_{i\in J}\abs{u_i-l_i}$, which in turn is a lower bound on $\vl(\varphi(\edgeab))$. To this end, we use the following simple geometric considerations.

The curve $\varphi(\edgeab)$ connects the points $\varphi(\mb{a})$ and $\varphi(\mb{b})$ and has a certain length. If we imagine that this curve becomes an inextensible string pinned to the points $\varphi(\mb{a})$ and $\varphi(\mb{b})$ in the plane, we can use it to construct an ellipse. We pull the string using a pencil to form a triangle. Then,
with a tip of the pencil, while keeping the string taut, we draw an ellipse. This is known as the gardener's construction.
An upper bound on the length of the minor axis of the described ellipse is calculated in the next observation, in which we also show that it upper bounds $\vl(\varphi(\edgeab))$.

\begin{obs}\label{obs:ellipse}
Let $\Gamma\colon[0,1]\rightarrow\R^2$ be a Lipschitz curve with endpoints $\varphi(\mathbf{a})$ and $\varphi(\mathbf{b})$ of length at most $Kb$. Then for every point $\mathbf{p}\in\Gamma$ the distance between the point $\mathbf{p}$ and the line passing through $\varphi(\mathbf{a})$ and $\varphi(\mathbf{b})$ is less than $Kb/2\cdot\sqrt{\alpha(2-\alpha)}$. As a consequence we have that $\vl(\Gamma)<Kb\cdot\sqrt{\alpha(2-\alpha)}$.
\end{obs}

\begin{proof}
By the arguments described above, the curve $\Gamma$ has to lie inside the ellipse with two focal points $\varphi(\mathbf{a}),\varphi(\mathbf{b})$ and sum of the distances from any point on the ellipse to its foci equal to $Kb$.

We can calculate the upper bound on the length of the semi-minor axis of this ellipse, because we know that $\lnorm{2}{\varphi(\mathbf{a})-\varphi(\mathbf{b})}>Kb(1-\alpha)$; using the Pythagorean theorem we have that the length of the semi-minor axis is less than $\sqrt{(Kb/2)^2-(Kb(1-\alpha)/2)^2}=Kb/2\cdot\sqrt{\alpha(2-\alpha)}$.

Now we prove that the length of the semi-minor axis upper bounds $\vl(\Gamma)/2$. Let us write $\Gamma_2$ for the $y$-coordinate of $\Gamma$. The definition of $\vl(\Gamma)$ implies that for every $\varepsilon>0$ we can find an increasing sequence $\{p_k\}_{k=0}^M\subset[0,1]$ such that $\sum_{k=1}^{M-1}\abs{\Gamma_2(p_{k+1})-\Gamma_2(p_k)}>(1-\varepsilon)\vl(\Gamma)$. We split the set of indices $\set{0}{M-1}$ into two collections $\mathcal{P}$ and $\mathcal{N}$. The collection $\mathcal{P}$ contains all $k\in\set{0}{M-1}$ such that $\Gamma_2(p_{k+1})-\Gamma_2(p_k)\geq 0$. Then we write $\mathcal{N}$ for the set $\set{0}{M-1}\setminus\mathcal{P}$.

We can assume that $\Gamma(p_0)=\varphi(\mathbf{a})$ and $\Gamma(p_M)=\varphi(\mathbf{b})$. Since the points $\varphi(\mathbf{a})$ and $\varphi(\mathbf{b})$ have the same $y$-coordinate, the sum $\sum_{k=1}^{M-1}{\Gamma_2(p_{k+1})-\Gamma_2(p_k)}$ is equal to zero, and thus
\begin{displaymath}
\sum_{k\in\mathcal{P}}{\Gamma_2(p_{k+1})-\Gamma_2(p_k)}=\sum_{k\in\mathcal{N}}\abs{\Gamma_2(p_{k+1})-\Gamma_2(p_k)}.
\end{displaymath}

We consider a collection of line segments $L_k:=\overline{\Gamma(p_k)\Gamma(p_{k+1})}$, where $k\in\set{0}{M-1}$. These line segments form a piecewise linear curve $\Lambda$ connecting the points $\varphi(\mathbf{a})$ and $\varphi(\mathbf{b})$ such that $\len(\Lambda)\leq\len(\Gamma)$ and $\vl(\Lambda)=\sum_{k=1}^{M-1}\abs{\Gamma_2(p_{k+1})-\Gamma_2(p_k)}$.

Finally, we permute the order in which the segments $L_k$ are connected, that is, we translate each of them and reconnect them in a different order creating an auxiliary piecewise linear curve $\Lambda'$; firstly, we take the line segments $L_i$ for all $i\in\mathcal{P}$ in arbitrary order, shift them and connect them so that they form a continuous piecewise linear curve starting at $\varphi(\mathbf{a})$ and heading only upwards. We end up at a point which we denote by $\mb{q}=(q_1, q_2)$. Secondly, we continue from $\mb{q}$ with all the segments $L_i$ such that $i\in\mathcal{N}$ in arbitrary order and connect them so that, after leaving $\mb{q}$, the curve $\Lambda'$ is going always downwards. Thus, we end up in $\varphi(\mathbf{b})$. The process is illustrated in Figure~\ref{fig:bound_on_lv}.

\begin{figure}[htb]
\centering
\includegraphics[width=110mm]{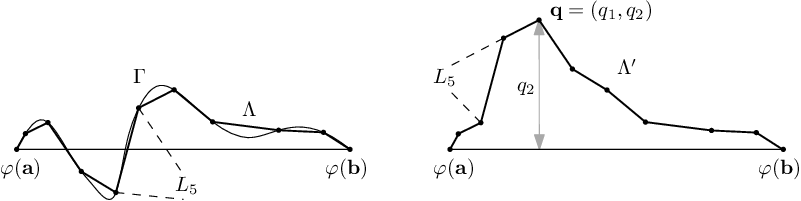}
\caption{\footnotesize{Bounding the vertical length of the curve $\Gamma$.}}
\label{fig:bound_on_lv}
\end{figure}

Clearly, $\len(\Lambda')=\len(\Lambda)$ and $\vl(\Lambda')=\vl(\Lambda)$. Since the $y$-coordi\-nate of the points $\varphi(\mathbf{a})$ and $\varphi(\mathbf{b})$ is $0$, we see that $\vl(\Lambda')=2q_2$. On the other hand, by the construction of $\Lambda'$, $q_2=\sum_{k\in\mathcal{P}}{\Gamma_2(p_{k+1})-\Gamma_2(p_k)}$.
By the first part of this observation applied to $\Lambda'$, we have that $q_2< Kb/2\cdot\sqrt{\alpha(2-\alpha)}$.
In total, we have proved that $(1-\varepsilon)\vl(\Gamma)<2q_2<Kb\cdot\sqrt{\alpha(2-\alpha)}$.
Letting $\varepsilon$ go to zero finishes the argument.
\end{proof}

Applying the observation above to $\varphi(\edgeab)$, we infer that $Kb\cdot\sqrt{\alpha(2-\alpha)}>\vl(\varphi(\edgeab))\geq \sum_{i\in J}\abs{u_i-l_i}$.

We define $\Omega:=\frac{\Delta N}{14}-Kb\cdot\sqrt{\alpha(2-\alpha)}$, which is the lower bound on $\frac{\Delta N}{14}-\sum_{i\in J}\abs{u_i-l_i}$. Using the inequality~(\ref{l_bound}), we get that $\Omega$ is also the lower bound on $\vl(P)$.

We now use the bound on $\vl(P)$ to obtain a lower bound on $\len(P)$, which will, eventually, lead to a contradiction.

We recall that $\lnorm{2}{\varphi(\mathbf{a})-\varphi(\mathbf{b})}>Kb(1-\alpha)$, and moreover, that $\varphi(\mathbf{a})$ and $\varphi(\mathbf{b})$ lie on the $x$-axis.

In order to compute the lower bound on $\len(P)$, we can start with the argumentation from the second part of the proof of Observation~\ref{obs:ellipse}. That is, for every $\varepsilon>0$ we can find a finite sequence of points on $P$ containing both $\varphi(\mathbf{a})$ and $\varphi(\mathbf{b})$ with the following property: if we connect these points with line segments $L_k$ in the order in which they lie on $P$, we get a piecewise linear curve $\Lambda$ such that $\vl(P)\geq\vl(\Lambda)>(1-\epsilon)\vl(P)$ and $\len(\Lambda)\leq\len(P)$. Then we permute the segments $L_k$ in the way described in the proof of Observation~\ref{obs:ellipse} and form an auxiliary piecewise linear curve $\Lambda'$ such that $\len(\Lambda')=\len(\Lambda)$ and $\vl(\Lambda')=\vl(\Lambda)$. Additionally, there is a point $\mb{q}\in\Lambda'$ such that between $\varphi(\mathbf{a})$ and $\mb{q}$ is $\Lambda'$ heading only upwards, while between $\mb{q}$ and $\varphi(\mathbf{b})$ it goes only downwards.

Subsequently, we make one further simplification---we consider the line segments $\overline{\varphi(\mb{a})\mb{q}}$ and $\overline{\mb{q}\varphi(\mb{b})}$ and denote the curve they form by $\Lambda''$. It is still true that $\vl(\Lambda'')=\vl(\Lambda)>(1-\epsilon)\vl(P)$ and $\len(\Lambda'')\leq\len(P)$. As before, we let $\varepsilon$ go to zero.
\begin{figure}[htb]
\centering
\includegraphics[width=110mm]{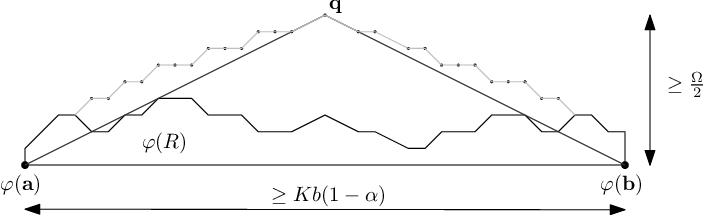}
\caption{\footnotesize{The lower bound on $\len(P)$.}}
\label{fig:5}
\end{figure}

The curve $\Lambda''$ is the shortest when it forms the legs of an isosceles triangle. Therefore, in order to calculate a lower bound on $\len(P)$, it suffices to calculate a lower bound on the length of the legs of an isosceles triangle of height $\vl(P)/2\geq\Omega/2$ with the base $\overline{\varphi(\mb{a})\varphi(\mb{b})}$. The whole idea is illustrated in Figure~\ref{fig:5}. Again, using the Pythagorean theorem we conclude
\begin{displaymath}
\len(P)>2\sqrt{\left(\frac{\Omega}{2}\right)^2+\left(\frac{Kb(1-\alpha)}{2}\right)^2}=\sqrt{\Omega^2+(Kb(1-\alpha))^2}.
\end{displaymath}
The curve $P$ consists of the images of $N+2$ line segments of length $h$, and thus $\len(P)\leq K(N+2)h$:
\begin{displaymath}
\Omega^2+K^2b^2(1-\alpha)^2\leq K^2(N+2)^2h^2.
\end{displaymath}
Substituting for $\Omega=\frac{\Delta N}{14}-Kb\cdot\sqrt{\alpha(2-\alpha)}$ and for $b=Nh$, we obtain, with some calculations,
\begin{displaymath}
\frac{\Delta^2N^2}{196}-\frac{\Delta N^2}{7}Kh\sqrt{\alpha(2-\alpha)}\leq K^2h^2(4N+4).
\end{displaymath}
Recall that $\Delta=\frac{1}{K}\cdot h\cdot \left(\frac{1-\gamma}{2\delta}-\frac{7\gamma}{1-14\alpha-14\delta}\right)$. As we already noted earlier, $\Delta$ is positive assuming an appropriate choice of $\gamma, \delta$ and $\alpha$. Defining $q=q(\alpha, \gamma, \delta):=\frac{1-\gamma}{2\delta}-\frac{7\gamma}{1-14\alpha-14\delta}>0$, we substitute for $\Delta=\frac{1}{K}hq$:
\begin{displaymath}
\frac{h^2q^2N^2}{196K^2}-\frac{Kh^2qN^2}{7K}\sqrt{\alpha(2-\alpha)}\leq K^2h^2(4N+4).
\end{displaymath}
Rearranging the above inequality and substituting $N_0\leq N$, we obtain:
\begin{displaymath}
q^2\leq 196K^4\bigg(\frac{4}{N_0}+\frac{4}{N_0^2}\bigg)+28qK^2\sqrt{\alpha(2-\alpha)}.
\end{displaymath}

From the last inequality we can see that for $N_0\rightarrow\infty$ and $\alpha\rightarrow 0$ the right hand side of the inequality converges to zero, while $q$ slightly grows up to its limit value $\frac{1-\gamma}{2\delta}-\frac{7\gamma}{1-14\delta}$. Consequently, for any $\delta\in(0,1/14)$ and $\gamma\in(0,1-14\delta)$ there is a choice of $\alpha=\alpha(\gamma, \delta, K)$ sufficiently small and of $N_0=N_0(\alpha, \gamma, \delta, K)$ large enough leading to a~contradiction.\qedhere
\end{proof} 

\paragraph{Acknowledgements.} I would like to thank my supervisor Professor Ji\v{r}\'i Matou\v{s}ek from Department of Applied Mathema\-tics, Faculty of Mathematics and Physics, Charles University in Prague, for valuable advice and help he has given me during the~work on my bachelor thesis~\cite{VK} and in writing this article. Honor his memory.

\def\bibname{References}

\bibliographystyle{alpha}

\bibliography{citations}

\end{document}